\documentclass[]{amsart}
\usepackage{amsfonts,latexsym,rawfonts,amsmath,amssymb, amsthm}
\usepackage{latexsym,lscape,rawfonts}
\usepackage{CJK}
\usepackage{mathrsfs, enumitem}
\usepackage[title]{appendix}
\usepackage{galois}
\usepackage[all,cmtip]{xy}
\usepackage{extarrows,color}
\usepackage{times}

\newenvironment{thmbis}[1]
{\renewcommand{\thethm}{\ref{#1}$ $}%
	\addtocounter{thm}{-1}%
	\begin{thm}}
	{\end{thm}}

\newtheorem{thm}{Theorem}[section]

\newtheorem*{fixed point criterion}{Fixed point criterion}
\newtheorem{cor}[thm]{Corollary}
\newtheorem{lem}[thm]{Lemma}
\newtheorem{prop}[thm]{Proposition}

\theoremstyle{definition}
\newtheorem{defn}[thm]{Definition}
\newtheorem{exam}[thm]{Example}
\newtheorem{ques}[thm]{Question}
\theoremstyle{remark}
\newtheorem{rem}[thm]{Remark}
\numberwithin{equation}{section}

\newcommand{\Z}{\mathbb Z}

\newcommand{\N}{\mathbb N}

\newcommand{\fix}{\mathrm{Fix}}

\newcommand{\ind}{\mathrm{ind}}

\newcommand{\rk}{\mathrm{rk}}

\newcommand{\aut}{\mathrm{Aut}}

\newcommand{\edo}{\mathrm{End}}

\newcommand{\id}{\mathrm{id}}

\newcommand{\F}{\mathbf{F}}      

\newcommand{\B}{\mathcal{B}}

\def\a{\alpha}
\def\b{\beta}
\def\g{\gamma}
\begin{document}

\title{Classification of aut-fixed subgroups in free-abelian times surface groups}

\author{Jialin Lei}
\address{School of Mathematics and Statistics, Xi'an Jiaotong University, Xi'an 710049, CHINA}
\email{leijialin0218@stu.xjtu.edu.cn}

\author{Peng Wang}
	\address{School of Mathematics and Statistics, Xi'an Jiaotong University, Xi'an 710049, CHINA}
	\email{wp580829@stu.xjtu.edu.cn}

\author{Qiang Zhang}
\address{School of Mathematics and Statistics, Xi'an Jiaotong University, Xi'an 710049, CHINA}
\email{zhangq.math@mail.xjtu.edu.cn}


\thanks{The authors are partially supported by NSFC (No. 12271385 and 12471066), the Shaanxi Fundamental Science Research Project for Mathematics and Physics (No. 23JSY027), and the Fundamental Research Funds for the Central Universities.}



\date{\today}
\begin{abstract}
In this paper, we are concerned with the direct product $G=\pi_1(\Sigma)\times \Z^k$ for $\Sigma$ a compact orientable surface with negative Euler characteristic, and give a complete classification of its fixed subgroups of automorphisms. As a corollary, we show that $G$ contains, up to isomorphism, infinitely many fixed subgroups of automorphisms if and only if $k\geq 2$, which is a contrast to that of hyperbolic groups. As an application on Nielsen fixed point theory, we provide a family of aspherical manifolds without Jiang's Bound Index Property. Moreover, we also give some results on the fixed subgroups of the direct product $H\times \Z^k$ for $H$ a non-elementary torsion-free hyperbolic group.
\end{abstract}
\maketitle


\section{Introduction}
For a group $G$, the \emph{rank} of $G$ denoted $\rk(G)$ is the minimal number of generators of $G$, and $\aut(G)$ (resp. $\edo(G)$) denotes the set of all automorphisms (resp. endomorphisms) of $G$. For any $\phi\in \edo(G)$, the \emph{fixed subgroup} of
$\phi$ is defined to be
 $$\fix\phi :=\{g\in G \mid \phi(g)=g\}.$$
The studies of fixed subgroups are of great current interest in various directions and related topics, see \cite{Car22, CL22, JZ24, Log21, RV21} etc.

It is well known that a subgroup of a free-abelian group $\Z^n$ is again free-abelian with rank $\leq n$, and hence, the fixed subgroup $\rk(\fix\phi)\leq \rk(\Z^n)$ for every $\phi\in \edo(\Z^n)$. For a free group $F_n$ of rank $n$, Bestvina and Handel \cite{BH92}, Imrich and Turner \cite{IT89} showed that $\rk(\fix\phi)\leq \rk(F_n)$ for every $\phi\in \edo(F_n)$.  The same inequality also holds for all endomorphisms of a \emph{surface group}, i.e., the fundamental group of a compact surface, see \cite{JWZ11}. Note that when a surface has boundary, its fundamental group is free. For finitely generated free-abelian groups and surface groups, subgroups are determined (perhaps up to finitely many options) up to isomorphism by their
ranks, then, finitely generated free-abelian groups and surface groups, all contain, up to isomorphism, only finitely many fixed subgroups of endomorphisms. Moreover, Shor \cite{Sh99} showed that the same finiteness also holds for the fixed subgroups of automorphisms of hyperbolic groups, and which was extended to monomorphisms by Lei and Zhang \cite{LZ23} recently. (For more information on the fixed subgroups in hyperbolic groups, see \cite{MO12, Neu92}). In \cite{HW04}, Hsu and Wise gave an example of a group, acting freely and cocompactly on a CAT(0) square complex, but with infinitely many types of fixed subgroups. Consequently, Shor's finiteness result does not hold, if the hyperbolicity condition is relaxed to either biautomaticity or nonpositive curvature.

In this paper, we consider the groups $\pi_1(\Sigma_g)\times \Z^k$ and $F_g\times \Z^k$, where $\Sigma_g$ denotes a closed orientable surface of genus $g$. Then the fundamental group $\pi_1(\Sigma_g)$ has a canonical presentation:
$$
\pi_1(\Sigma_g)=\langle a_1,b_1,\cdots ,a_g,b_g\mid [a_1,b_1]\cdots [a_g,b_g]=1 \rangle,
$$
for $[a, b]=aba^{-1}b^{-1}$.\\

\noindent\textbf{Notations.}  Throughout the paper, let $g, k, r, s, t\ldots$ denote some natural numbers, and hence they are not $\aleph_0$ (i.e., the cardinality of the integer set $\Z$), unless stated otherwise. Let $\lfloor \frac g2 \rfloor$ denote the biggest integer $\leq \frac g2$.\\

The main results of this paper are the following.

\begin{thm}\label{main thm1}
Let $G=\pi_1(\Sigma_g)\times \Z^k$ or $F_g\times \Z^k$ for $g, k\geq 2$. Then
$$\{\rk(\fix\phi)\mid \phi\in \aut(G)\}=\N\cup\{\aleph_0\},$$
and hence, $G$ contains, up to isomorphism, infinitely many fixed subgroups of automorphisms.
\end{thm}

In contrast, both of $\pi_1(\Sigma_g)\times \Z$ and $F_g\times \Z$ contain, up to isomorphism, only finitely many (more precisely, $5g+2$ and $2g+\lfloor \frac g2\rfloor +2$ respectively) fixed subgroups of automorphisms,  see Corollary \ref{main cor2'} and Corollary \ref{main cor2' surf}.

To prove the main results, we study when a subgroup can be realised as a fixed subgroup of an automorphism.

\begin{defn}
A subgroup $H\leq G$ is called \emph{aut-fixed} in $G$, if there exists an automorphism $\phi\in\aut(G)$ such that $\fix\phi=H$; and $H$ is called \emph{aut-fixed up to isomorphism}, if there exists $\phi\in\aut(G)$ such that $\fix\phi\cong H$ (here ``$\cong$" means isomorphism).
\end{defn}

In the paper \cite{DV13}, Delgado and Ventura produced an algorithm to determine when fixed subgroups of free times free-abelian groups are finitely generated and in that
case output a basis. In this paper, we give a complete characterization of aut-fixed subgroups as follows.

\begin{thm}\label{main thm3}
For $G=F_g\times \Z^k$ ($g, k\geq 2$), all subgroups but $F_t\times \Z^k$ ($t> g$) and $F_{\aleph_0}$, are aut-fixed up to isomorphism; or equivalently, a subgroup $H\leq G$ is aut-fixed if and only if it has one of the following isomorphic classes:
\begin{enumerate}
\item $F_t\times \Z^k, ~~0\leq t\leq g;$
\item $F_t\times \Z^s, ~~t\geq 0, ~~0\leq s\leq k-1;$
\item $F_{\aleph_0}\times \Z^s, ~~1\leq s\leq k.$
\end{enumerate}
\end{thm}

\begin{thm}\label{main thm4}
For $G=\pi_1(\Sigma_g)\times \Z^k$ ($g, k\geq 2$), all subgroups but  $F_t\times \Z^k$ ($t\geq 2g$),  $\pi_1(\Sigma_r)\times \Z^k$ ($r> g$) and $F_{\aleph_0}$, are aut-fixed up to isomorphism; or equivalently, a subgroup $H\leq G$ is aut-fixed if and only if it has one of the following isomorphic classes:
\begin{enumerate}
   \item $F_t\times \Z^s$ for $t\geq 0, ~0\leq s\leq k$ with $t< 2g$ if $s=k$;
  \item $\pi_1(\Sigma_{m(g-1)+1})\times \Z^s$ for $m\geq 1, ~0\leq s\leq k$ with $m=1$ if $s=k$;
  \item $F_{\aleph_0}\times \Z^s$ for $1\leq s \leq k$.
\end{enumerate}
\end{thm}

As an application on Nielsen fixed point theory, we provide a family of aspherical manifolds without Jiang's Bound Index Property (see Subsection \ref{subsect Nilesen pt} for definitions and more details).

\begin{thm}\label{high-dim cout-examp.}
Let $S^1$ be the circle and $T^k$ the $k$-dimension torus. Then
\begin{enumerate}
  \item the $3$-manifold $\Sigma_g\times S^1 (g\geq 2)$ has BIPH, but does not have BIP;
  \item the manifold $\Sigma_g\times T^k (g, k\geq 2)$ does not have BIPH, and hence does not have BIP.
\end{enumerate}
\end{thm}

Although every hyperbolic group contains, up to isomorphism, only finitely many fixed subgroups of automorphism, by a recent deep result of Lazarovich \cite{Laz23}: \emph{every non-elementary hyperbolic group is finite index rigid} (that is, it does not contain isomorphic finite index subgroups of different indices), we have:

\begin{thm}\label{main thm5}
Let $H$ be a non-elementary torsion-free hyperbolic group. Then $G=H\times \Z^k$ ($k\geq2$) contains, up to isomorphism, infinitely many fixed subgroups of automorphism if and only if the first betti number of $H$ is non-zero.
\end{thm}

The main idea and technique used in this paper, is to construct some special automorphisms, whose fixed subgroups satisfy certain conditions.
The paper is organized as follows. In Section \ref{sect 2}, we introduce some important facts on fixed subgroups in surface groups, and give a complete characterization of subgroups of $\pi_1(\Sigma_g)\times \Z^k$ for later use. In Section \ref{sect 2'}, we introduce three useful lemmas to facilitate our proof of the main theorems. In Section \ref{sect 3}, we prove the main results for fixed subgroups in $F_g\times \Z^k$.  In Section \ref{sect 4}, we first prove the main results for $\pi_1(\Sigma_g)\times \Z^k$, which are parallel to that for $F_g\times \Z^k$, and then give some discussions on Jiang's (Bounded Index Property) Question \ref{Jiang's question} for aspherical manifolds in Nielsen fixed point theory. Finally in Section \ref{sect 5}, we prove Theorem \ref{main thm5} for hyperbolic groups.

\section{Preliminary}\label{sect 2}

In general, subgroups of direct products of abstract groups are complicated, see \cite{Bri23, BH07, BHMS09} for some new and deep results. In this section, we will give a complete characterization of subgroups of $G=\pi_1(\Sigma_g)\times \Z^k$ or $F_g\times \Z^k$, where $\Z^k$ is viewed as an addition group and an element $v\in\Z^k$ is a column vector.

\subsection{Subgroups}
It is well known that a subgroup of $\Z^k$ is isomorphic to $\Z^s$ for some $s\leq k$, and a subgroup of a free group is also free. Moreover,

\begin{lem}[Schreier's formula]\label{F_2 subgp}
Let $H$ be a subgroup of $F_g$ ($g\geq 2$) with index $[F_g : H]=m$. Then $$H\cong F_{m(g-1)+1}.$$
Moreover, $H\cong F_{\aleph_0}$ if $H$ is nontrivial and normal with index $[F_g : H]$ infinite.
\end{lem}

For closed surface groups, we have the following well-known result.

\begin{lem}\label{subgp of surface}
Let $H$ be a subgroup of $\pi_1(\Sigma_g)$ ($g\geq 2$). Then
\begin{enumerate}
  \item $H\cong F_t$ ($t\geq 0$) or $F_{\aleph_0}$, if the index $[\pi_1(\Sigma_g):H]=\infty$;
  \item $H\cong \pi_1(\Sigma_{m(g-1)+1})$ if the index $[\pi_1(\Sigma_g):H]=m<\infty$.
\end{enumerate}
Moreover, $H\cong F_{\aleph_0}$ if $H$ is nontrivial and normal with index $[\pi_1(\Sigma_g) : H]$ infinite.
\end{lem}

\begin{proof}
Let $X_H$ be the covering space of $\Sigma_g$ with fundamental group $\pi_1(X_H)=H$. Then

(1) if $[\pi_1(\Sigma_g):H]=\infty$, then $X_H$ is a noncompact surface, and hence $H$ is a free group \cite{Ja70} (it is clear that $H$ is a free group when $X_H$ is a punctured surface, and when $X_H$ is an infinite surface, for example, one can see \cite[pp.142-144]{St93} that $H$ is also a free group). If $H$ is nontrivial and normal with index $[\pi_1(\Sigma_g) : H]$ infinite, then $H\cong F_{\aleph_0}$. (In fact, a similar result has been proven to hold for limit groups by Bridson and Howie \cite{BH07'}.)

(2) if $[\pi_1(\Sigma_g):H]=m$, then $X_H$ is a closed orientable surface $\Sigma_r$ with Euler characteristic
$$\chi(\Sigma_r)=2-2r=m(2-2g).$$ Therefore, $r=m(g-1)+1$ and hence $H\cong \pi_1(\Sigma_{m(g-1)+1})$.
\end{proof}

By Delgado and Ventura \cite[Corollary 1.7]{DV13} which gave a characterization of subgroups of $F_g\times \Z^k$, we have

\begin{lem}[Delgado-Ventura]\label{DV subgp}
Let $G=F_g\times \Z^k$ ($g\geq 2, k\geq 1$). Then

(1) Every finitely generated subgroup of $G$ has the following form:
$$F_t\times \Z^s, ~~t\geq 0, ~0\leq s\leq k;$$

(2) Every infinitely generated subgroup of $G$ has the following form:
$$F_{\aleph_0}\times \Z^s, ~~ ~0\leq s\leq k.$$
\end{lem}

For $\pi_1(\Sigma_g)\times \Z^k$, we prove a parallel characterization of subgroups as follows.

\begin{lem}\label{surface subgp}
Let $G=\pi_1(\Sigma_g)\times \Z^k$ ($g\geq 2, k\geq 1$). Then

(1) Every finitely generated subgroup of $G$ has one of the following forms:
$$F_t\times \Z^s, ~\quad \pi_1(\Sigma_{m(g-1)+1})\times \Z^s, \quad t\geq 0, ~m\geq 1, ~0\leq s\leq k; $$

(2) Every infinitely generated subgroup of $G$ has the following form:
$$F_{\aleph_0}\times \Z^s, ~~ ~0\leq s\leq k.$$
\end{lem}

\begin{proof}
Let $H$ be a subgroup of $G=\pi_1(\Sigma_g)\times \Z^k$. Now we consider the projection $$p:  \pi_1(\Sigma_g)\times \Z^k\to  \pi_1(\Sigma_g), \quad p(u, v)=u,$$
and let $p_H:  H \to p(H)\leq \pi_1(\Sigma_g)$ be the restriction of $p$ on $ H $.
Then we have the natural short exact sequence
$$1\to \ker(p_H)\hookrightarrow H  \xrightarrow{p_H} p(H)\to 1.$$
Note that $\ker(p_H)\cong \Z^s$ for some $s\leq k$ is a subgroup of $\Z^k$, and
$p(H)$ is a subgroup of $\pi_1(\Sigma_g)$.

If $p(H)$ is free, we can define a monomorphism $\iota:p(H)\to  H$
sending each element of a chosen free basis for $p(H)$ back to an arbitrary preimage, i.e, $p\comp \iota$ is the identity of $p(H)$. If $p(H)\leq \pi_1(\Sigma_g)$ is not free, then $p(H)\cong\pi_1(\Sigma_n)$ for $n={m(g-1)+1}$ and $m=[H : p(H)]<\infty$. Pick a canonical presentation
$$p(H)=\langle a_1,b_1,\cdots ,a_n,b_n\mid [a_1,b_1]\cdots [a_n,b_n]=1 \rangle.$$
Note that all the elements $a_i, b_i\in \pi_1(\Sigma_g)$. Now for each $i=1,\ldots, n$, pick any two coordinate components $s_i\in p^{-1}(a_i)\cap \Z^k$ and $t_i\in p^{-1}(b_i)\cap \Z^k$, and define
\begin{eqnarray}
\iota: p(H)&\to& H\subset \pi_1(\Sigma_g)\times \Z^k\notag\\
a_i &\mapsto& (a_i, ~s_i)\notag,\\
b_i &\mapsto& (b_i, ~t_i)\notag.
\end{eqnarray}
Notice $\iota$ is well defined because $[s_1,t_1]\cdots [s_n,t_n]=0\in \Z^k$, and $p\comp \iota$ is also the identity of $p(H)$.
It implies  that
the above exact sequence is split.

Note that $\ker(p_H)\cong \Z^s$ is abelian, by straightforward calculations, we can see the following
map is an isomorphism:
\begin{eqnarray}
\Psi:  H  &\to& \iota(p(H))\times \ker(p_H)\notag\\
 h&\mapsto& \big((\iota\comp p)(h), ~~h\cdot(\iota\comp p)(h^{-1})\big)\notag.
\end{eqnarray}
Therefore,
$$
H =\iota(p(H))\times \ker(p_H)\cong p(H)\times\Z^s (s\leq k).
$$
Note that $p(H)$ is a subgroup of $\pi_1(\Sigma_g)$, then the conclusion follows from Lemma \ref{subgp of surface}.
\end{proof}

\subsection{Aut-fixed subgroups}

Recall that a group $H$ is \emph{aut-fixed up to isomorphism} in $G$, if there exists $\phi\in\aut(G)$ such that $\fix\phi\cong H$. In particular, if $\fix\phi= H$, we call $H$ \emph{aut-fixed} in $G$.
Now we show that aut-fixed subgroups have multiplicativity.

\begin{lem}[Multiplicativity of aut-fixed subgroup]\label{product auo-fixed sbgp}
Let $H_i$ be an aut-fixed subgroup (up to isomorphism) of $G_i$ for $i=1,\ldots, n$. Then $H_1\times\cdots\times H_n$ is an aut-fixed subgroup (up to isomorphism) of $G_1\times\cdots\times G_n$.
\end{lem}

\begin{proof}
Since $H_i$ is an aut-fixed subgroup (up to isomorphism) of $G_i$, we have $H_i\cong \fix\phi_i$ for some $\phi_i\in \aut(G_i)$. Thus \begin{eqnarray}
\phi_1\times\cdots\times \phi_n: G_1\times\cdots\times G_n&\to& G_1\times\cdots\times G_n,\notag\\
(u_1,\ldots, u_n)&\mapsto& (\phi_1(u_1), \ldots, \phi_n(u_n))\notag
\end{eqnarray}
is an automorphism of $G_1\times\cdots\times G_n$, and
$$H_1\times\cdots\times H_n\cong \fix\phi_1\times\cdots\times \fix\phi_n\cong \fix(\phi_1\times\cdots\times \phi_n).$$
\end{proof}

Note that some subgroups of $\Z^k$ are not aut-fixed, e.g. $2\Z$ is not aut-fixed in $\Z$, but for aut-fixed up to isomorphism, we have

\begin{lem}\label{fixed sbgp in Z^K}
Every subgroup of $\Z^k$ is aut-fixed up to isomorphism.
\end{lem}

\begin{proof}
Let $H$ be a subgroup of $\Z^k$. Then $H\cong \Z^s$ for some $0\leq s\leq k$. Pick $\phi_i$ the identity of $\Z$ for $i=1,\ldots,s$ and $\phi_j=-\id$ of $\Z$, that is, $\phi_j(a)=-a$ for $j=s+1,\ldots n$, then $\fix(\phi_1\times\cdots\times \phi_n)\cong \Z^s.$
\end{proof}

\begin{prop}[Bestvina-Handel]\label{fixed subgp in free gp}
A subgroup $H\leq F_n$ is aut-fixed up to isomorphism if and only if $H\cong F_t$ for $0\leq t\leq n$.
\end{prop}

\begin{proof}
Let $\phi$ be an automorphism of $F_n$. Note that a subgroup of a free group is again free, we have $\fix\phi\cong F_t$ for some $t\leq n$ by Bestvina and Handel's deep result. Conversely, for any $t\leq n$, let
$$\phi_t: F_n\to F_n=\langle a_1, \ldots, a_n\rangle, \quad a_i\mapsto a_i, ~i\leq t; ~a_j\mapsto a_j^{-1}, ~j>t.$$
Then $\fix\phi_t=\langle a_1, \ldots, a_t\rangle\cong F_t$.
\end{proof}

In \cite{JWZ11}, Jiang, Wang and Zhang  studied the fixed subgroups of surface automorphisms, and showed that $\rk(\fix\phi)\leq 2g$ for every automorphism $\phi\in\aut(\pi_1(\Sigma_g))$, by using such an assertion:

\emph{For a standard homeomorphism $f: \Sigma_g\to\Sigma_g$ in the sense of Nielsen-Thurston classification, each Nielsen path of $f$ can be deformed (rel. endpoints) into the fixed point set $\fix f$.}

Summarise ``Proof of Corollary T and  Proof of Theorem 3.2 for pseudo-Anosov $\varphi$" in \cite[pp. 2304-2306]{JWZ11}, we have a key lemma as follows.

\begin{lem}[Jiang-Wang-Zhang]\label{JWZ mafandeyinli}
(1) Let $\Sigma$ be a compact surface with negative Euler characteristic, and let $\phi: \pi_1(\Sigma, x)\to \pi_1(\Sigma, x)$ be an automorphism induced by a pseudo-Anosov homeomorphism  $f$ of $\Sigma$. Then, $\fix\phi=1$ if $\Sigma$ is closed, while $\fix\phi=\langle c\rangle\cong \Z$ if the boundary $\partial \Sigma$ is nonempty and $f$ fixes some component $c$ of $\partial \Sigma$ pointwise with the base point $x\in c$.

(2) Let $\Sigma_g=\Sigma'\cup\Sigma''$ be a splitting of the closed surface $\Sigma_g$ of genus $g\geq 2$, where $\Sigma'$ and $\Sigma''$ are both compact subsurfaces of $\Sigma_g$ with negative Euler characteristics and with the same boundaries $\partial\Sigma'=\partial\Sigma''$. If $f'$ is the identity of $\Sigma'$ and $f''$ is a pseudo-Anosov homeomorphism of $\Sigma''=\overline{\Sigma_g-\Sigma'}$ fixing the boundary $\partial\Sigma''$. Then $f=f'\cup f''$ is a homeomorphism of $\Sigma_g$ inducing an automorphism $\phi\in\aut(\pi_1(\Sigma_g))$ with
$$\fix\phi=\pi_1(\Sigma')\cong F_{t} ~~(2\leq t\leq  2g-2).$$
\end{lem}

By using Lemma \ref{JWZ mafandeyinli}, we can show a result on aut-fixed subgroups of surface groups.

\begin{prop}\label{fixed subgp in surface gp}
A subgroup $H\leq \pi_1(\Sigma_g)$ ($g\geq 2$) is aut-fixed up to isomorphism if and only if $H\cong \pi_1(\Sigma_g)$ or $H\cong F_t$ ($0\leq t< 2g$). More precisely, for a canonical presentation
$$\pi_1(\Sigma_g)=\langle a_1,b_1,\cdots ,a_g,b_g\mid [a_1,b_1]\cdots [a_g,b_g]=1 \rangle,$$
and for any $1\leq k<g$, there exists $\phi_k, \psi_k\in \aut(\pi_1(\Sigma_g))$ such that
$$\fix\phi_k=\langle a_1, b_1, \ldots, a_k, b_k\rangle\cong F_{2k};$$
$$\fix\psi_k=\langle a_1, b_1, \ldots, a_{k}, b_{k}, a_{k+1}\rangle\cong F_{2k+1}.$$
\end{prop}

\begin{proof}
The ``only if " part is clear by \cite[Theorem 1.2]{JWZ11} and Lemma \ref{subgp of surface}. Below, we prove the ``if " part.

The cases of $H=\pi_1(\Sigma_g)$ and $H=1$ are trivial, because $\fix(\id)=\pi_1(\Sigma_g)$ and $\fix\phi=1$ for $\phi\in\aut(\pi_1(\Sigma_g))$ induced by a pseudo-Anosov homeomorphism of the closed surface $\Sigma_g$.

For the case of $H\cong \Z$, first, pick a simple closed curve $c\subset \Sigma_g$ which separates $\Sigma_g$ into two components $\Sigma'$ and $\Sigma''$ with boundaries $\partial\Sigma'=\partial\Sigma''=c$. Then, take two pseudo-Anosov homeomorphisms $f'$ and $f''$ of $\Sigma'$ and $\Sigma''$ respectively, such that
$$f'|_c=f''|_c=\id: c\to c.$$
By Lemma \ref{JWZ mafandeyinli}, the automorphisms induced by $f'$ and $f''$ both have fixed subgroups $\langle c \rangle\cong \Z$, and hence we have a homeomorphism $f=f'\cup f''$ of $\Sigma_g$ inducing an automorphism $\phi\in\aut(\pi_1(\Sigma_g))$ with $\fix\phi=\langle c \rangle\cong \Z$.

For  $H\cong F_{2k} ~(1\leq k<g)$, note that
$$\langle a_1, b_1, \ldots, a_k, b_k\rangle=\pi_1(\Sigma'_k)\cong F_{2k}$$
where $\Sigma'_k$ is a compact subsurface of the closed surface $\Sigma_g$ and homeomorphic to a closed surface $\Sigma_k$ with an open disk removed. Then pick $f'$ the identity of $\Sigma'_k$ and $f''$ a pseudo-Anosov homeomorphism of $\Sigma''=\overline{\Sigma_g-\Sigma'_k}$ fixing the boundary $\partial\Sigma''=\partial\Sigma'_k$. Then $f=f'\cup f''$ is a homeomorphism of $\Sigma_g$ inducing an automorphism $\phi_k\in\aut(\pi_1(\Sigma_g))$ with $\fix\phi_k=\pi_1(\Sigma'_k)\cong F_{2k}.$

For $H\cong F_{2k+1} ~(1\leq k<g-1)$, let
$$\langle a_1, b_1, \ldots, a_k, b_k, a_{k+1}\rangle=\pi_1(\Sigma'_k)\cong F_{2k+1}$$
where $\Sigma'_k$ is a compact subsurface of the closed surface $\Sigma_g$ and homeomorphic to a closed surface $\Sigma_k$ with two open disks removed. Then pick $f'$ the identity of $\Sigma'_k$ and $f''$ a pseudo-Anosov homeomorphism of $\Sigma''=\overline{\Sigma_g-\Sigma'_k}$ fixing the boundary $\partial\Sigma''=\partial\Sigma'=S^1\sqcup  S^1$, the disjoint union of two circles. Then $f=f'\cup f''$ is a homeomorphism of $\Sigma_g$ inducing an automorphism $\psi_k\in\aut(\pi_1(\Sigma_g))$ with $\fix\psi_k=\pi_1(\Sigma'_k)\cong F_{2k+1}.$

For the rest case of $H= \langle a_1,b_1,\cdots, a_{g-1}, b_{g-1}, a_g\rangle \cong F_{2g-1}$, we take $\psi\in\aut(\pi_1(\Sigma_g))$ as $\psi(b_g)=b_ga_g$ and $\psi$ fixes the other generators. Then by straightforward computations, $\fix\psi=H$. (See \cite[Example 4.7]{LZ23}).
\end{proof}

\subsection{Automorphisms}

Now we give a standard form for every automorphism of $G=\pi_1(\Sigma_g)\times \Z^k$ or $F_g\times \Z^k$.

\begin{lem}\label{standerd form}
Let $\phi$ be an automorphism of $H\times \Z^k$ ($k\geq 1$) where $H$ is a non-elementary torsion-free hyperbolic group. Then for any $(u, v)\in H\times \Z^k$,
$$\phi(u, v)=(\a(u), \g(u)+\mathcal{L}(v)),$$
where $\a: H\xrightarrow{\cong} H$ and $\mathcal{L}: \Z^k\xrightarrow{\cong}\Z^k$ are automorphisms, and $\g: H\to \Z^k$ is a homomorphism.
\end{lem}

\begin{proof}
Note that $H$ has trivial center and $\Z^k$ is nilpotent, then our conclusion follows from \cite[Lemma 2.4]{N20}.
\end{proof}

\begin{lem}\label{iso type}
Let $g_i\geq 2$ and $k_i\geq 1$ for $i=1,2$. Then
\begin{enumerate}
 \item $\pi_1(\Sigma_{g_1})\times \Z^{k_1}\not\cong F_{g_2}\times \Z^{k_2}$;
  \item $\pi_1(\Sigma_{g_1})\times \Z^{k_1}\cong \pi_1(\Sigma_{g_2})\times \Z^{k_2}$ if and only if $g_1=g_2$ and $k_1=k_2$;
  \item $F_{g_1}\times \Z^{k_1}\cong F_{g_2}\times \Z^{k_2}$ if and only if $g_1=g_2$ and $k_1=k_2$;
  \item $\rk(\pi_1(\Sigma_{g})\times \Z^{k})=2g+k$ and $\rk(F_{g}\times \Z^{k})=g+k$ for any $g,k\geq 0$.
\end{enumerate}
\end{lem}

\begin{proof}
Let $H_i=\pi_1(\Sigma_{g_i})$ or $F_{g_i}$ ($g_i\geq 2$) for $i=1,2$. Note that all the centers $\mathcal C(H_i)$ are trivial and
$$\mathcal C(H_i\times \Z^{k_i})=\mathcal C(H_i)\times \mathcal C(\Z^{k_i})=\Z^{k_i}.$$
Therefore, $k_1=k_2$ if $H_1\times \Z^{k_1}\cong H_2\times \Z^{k_2}$. Moreover, if $\phi: H_1\times \Z^{k_1}\to H_2\times \Z^{k_2}$ is an isomorphism, then
$$\phi(\Z^{k_1})=\phi(\mathcal C(H_1\times \Z^{k_1}))=\mathcal C(H_2\times \Z^{k_2})=\Z^{k_2},$$
and hence
$$H_1=\frac{H_1\times \Z^{k_1}}{\mathcal C(H_1\times \Z^{k_1})}\cong \frac{H_2\times \Z^{k_2}}{\mathcal C(H_2\times \Z^{k_2})}=H_2.$$
So items (1)-(3) hold. Item (4) follows from the fact that the abelianization $\pi_1(\Sigma_{g})^\mathrm{ab}\cong \Z^{2g}$ and $(F_{g})^\mathrm{ab}\cong \Z^{g}$.
\end{proof}

\section{Three useful lemmas}\label{sect 2'}

In order to facilitate our proof of the main theorems later, we introduce three useful lemmas in this section.

Throughout this section, let $H=F_g$ or $H=\pi_1(\Sigma_g)$ for $g\geq 2 $ and $\phi$ be an automorphism of $G=H\times \Z^k$ ($k\geq 1$). Note that $H$ is a non-elementary torsion-free hyperbolic group. Then by Lemma \ref{standerd form}, the automorphism $\phi\in\aut(G)$ has the form:
$$\phi(u, v)=(\a(u), \g(u)+\mathcal{L}(v)), \quad\forall (u, v)\in H\times \Z^k,$$
where $\a: H\xrightarrow{\cong}  H$ and $\mathcal{L}: \Z^k\xrightarrow{\cong}\Z^k$ are automorphisms, and
$\g:  H\to \Z^k$ is a homomorphism.

Now we consider the projection $$p:  H\times \Z^k\to  H, \quad p(u, v)=u,$$
and let $p_\phi: \fix\phi\to p(\fix\phi)\leq H$ be the restriction of $p$ on $\fix\phi$.
Then we have a natural short exact sequence
$$1\to \ker(p_\phi)\hookrightarrow\fix\phi \xrightarrow{p_\phi} p(\fix\phi)\to 1,$$
where
$$
p(\fix\phi)=\{u\in \fix\a \mid \exists v=\g(u)+\mathcal{L}(v)\in \Z^k \}\vartriangleleft \fix\a,
$$
is a normal subgroup of  $\fix\a$, and
\begin{eqnarray}\label{equ 6}
\ker(p_\phi)&=&\{(1, v)\in H\times\Z^k\mid v=\mathcal{L}(v)\}\notag\\
&=&\{v\in\Z^k\mid v=\mathcal{L}(v)\}\cong \Z^s,\notag
\end{eqnarray}
for some $s\leq k$. Now we have two cases:

\textbf{Case ($H=F_g$).} Since the subgroup $p(\fix\phi)$ of $F_g$ is free, we can define a monomorphism
$$\iota:p(\fix\phi)\to \fix\phi$$
sending each element of a chosen free basis for $p(\fix\phi)$ back to an arbitrary preimage, i.e, $p\comp \iota$ is the identity of $p(\fix\phi)$. It implies  that
the above exact sequence is split. Note that $\ker(p_\phi)\cong \Z^s$ is abelian, by straightforward calculations, we can see the following
map is an isomorphism:
\begin{eqnarray}
\Psi: \fix\phi &\to& \iota(p(\fix\phi))\times \ker(p_\phi)\notag\\
 h&\mapsto& \big((\iota\comp p)(h), ~~h\cdot(\iota\comp p)(h^{-1})\big)\notag.
\end{eqnarray}
Therefore,
$$\fix\phi=\iota(p(\fix\phi))\times \ker(p_\phi)\cong p(\fix\phi)\times\Z^s (s\leq k).$$

\textbf{Case ($H=\pi_1(\Sigma_g)$).} In this case, by the proof of Lemma \ref{surface subgp}, we also have
\begin{eqnarray}
\fix\phi &=&\{(u,v)\in  \pi_1(\Sigma_g)\times \Z^k\mid u\in\fix\a, ~v=\g(u)+\mathcal{L}(v)\}\nonumber\\
&\cong& p(\fix\phi)\times \ker(p_\phi)\nonumber\\
&\cong& p(\fix\phi)\times\Z^s (s\leq k).\nonumber
\end{eqnarray}

In conclusion, we have proven the following:
\begin{lem}\label{adding lem 1}
$\fix\phi\cong p(\fix\phi)\times \ker(p_\phi)\cong p(\fix\phi)\times\Z^s$ for some $s\leq k.$
\end{lem}

Moreover, we have:

\begin{lem}\label{adding lem 2}
If $\mathcal{L}: \Z^k\xrightarrow{\cong}\Z^k$ has no eigenvalue $1$, then $p(\fix\phi)$ is finitely generated and hence $p(\fix\phi)\not\cong F_{\aleph_0}$.
\end{lem}

\begin{proof}
Recall that
$$
p(\fix\phi)=\{u\in \fix\a \mid \exists v=\g(u)+\mathcal{L}(v)\in \Z^k\}.
$$
Since $\mathcal L$ has no eigenvalue $1$, we have
$$d:=\det(\id-\mathcal{L})\neq 0,$$
where $\mathcal L$ is viewed as a matrix. Let us consider the linear equation,
$$X=\g(u)+\mathcal{L}X.$$
It always has a unique real root: $X_u=(\id-\mathcal{L})^{-1}\g(u)$ for $u\in \fix\a$, but we need an integer root. Notice however
$d(\id-\mathcal{L})^{-1}$ is an integer matrix, and hence,
\begin{eqnarray}\label{eq 50}
dX_u&=&d(\id-\mathcal{L})^{-1}\g(u)\notag\\
&=&(\id-\mathcal{L})^{-1}d\g(u)\notag\\
&=&(\id-\mathcal{L})^{-1}\g(u^d)\in \Z^k.\notag
\end{eqnarray}
Therefore, we have
$$\Gamma_d:=\langle \{u^d,[u,v]\mid u,v\in\fix\a\}\rangle \subseteq p(\fix\phi),$$
and the index
$$[\fix\a : p(\fix\phi)]\leq[\fix\a : \Gamma_d]\leq |d|^r,~~r=\rk(\fix\alpha).$$

If $H=F_g$, then $\a\in\aut(F_g)$ and $\fix\a=F_r$ for some $r\leq g$ by Bestvina-Handel's deep result, and hence $p(\fix\phi)$ is a free group of finite rank. So
$p(\fix\phi)\not\cong F_{\aleph_0}$ in this case.

If $H=\Sigma_g$, then $\a\in\aut(\pi_1(\Sigma_g))$ and $\fix\a=\pi_1(\Sigma_g)$ or $\fix\a=F_r (r<2g)$ by Proposition \ref{fixed subgp in surface gp}, and hence
we also have $p(\fix\phi)\not\cong F_{\aleph_0}$ by Lemma \ref{F_2 subgp} and Lemma \ref{subgp of surface}.
\end{proof}

Recall that $\a\in \aut(H)$, $\mathcal{L}\in \aut(\Z^k)$ and $\g: H\to \Z^k$ is a homomorphism.

\begin{lem}\label{adding lem 3}
If $\mathcal{L}: \Z^k\xrightarrow{\cong}\Z^k$ is the identity and $\g(\fix\a)$ is a nontrivial subgroup of $\Z^k$, then $$p(\fix\phi)\cong F_{\aleph_0} ~\mathrm{or} ~1.$$
\end{lem}

\begin{proof}
Since $\mathcal{L}$ is the identity, we have
\begin{eqnarray}\label{equ 9''}
p(\fix\phi)&=&\{u\in \fix\a \mid \exists v=\g(u)+\mathcal{L}(v)\in \Z^k\}\\
&=&\{u\in  \fix\a \mid \g(u)=0\}\vartriangleleft \fix\a.\notag
\end{eqnarray}
Moreover, note that $\g(\fix\a)$ is a nontrivial subgroup of $\Z^k$, then it is infinite, in fact, the order $|\g(\fix\a)|=\aleph_0$. Now consider the epimorphism
$$\g: \fix\a \twoheadrightarrow \g(\fix\a)\leq \Z^k.$$
Then Eq. (\ref{equ 9''}) implies the index
$$
[\fix\a : p(\fix\phi)]=|\g(\fix\a)|=\aleph_0.
$$
Namely, $p(\fix\phi)$ is an infinite-index normal subgroup of $\fix\a$. Whenever $H=F_g$ or $H=\pi_1(\Sigma_g)$, we always have
$$p(\fix\phi)\cong F_{\aleph_0} ~ \mathrm{or} ~ 1,$$
by combining Lemma \ref{F_2 subgp}, Lemma \ref{subgp of surface}, Proposition \ref{fixed subgp in free gp} and Proposition \ref{fixed subgp in surface gp}.
\end{proof}

\section{Fixed subgroups in $F_g\times \Z^k$}\label{sect 3}

In this section, we will study the fixed subgroup of a general automorphism of $G=F_g\times \Z^k$ and prove Theorem \ref{main thm3}.

\subsection{Non-fixed subgroups in $F_g\times \Z^k$}

\begin{prop}\label{fixed can not big than Z^k}
Let $\phi$ be an automorphism of $G=F_g\times \Z^k$ ($g\geq 2, k\geq 1$). Then
$$\fix\phi\not\cong F_{\aleph_0}, ~F_t\times \Z^k ~(t>g).$$
Namely, $F_{\aleph_0}$ and $F_t\times \Z^k$ ($t>g$) are not aut-fixed up to isomorphism in $G=F_g\times \Z^k$ ($g\geq 2, k\geq 1$).
\end{prop}

\begin{proof}
By Lemma \ref{standerd form}, an automorphism $\phi\in\aut(G)$ has the form:
$$\phi(u, v)=(\a(u), \g(u)+\mathcal{L}(v)), \quad\forall (u, v)\in F_g\times \Z^k,$$
where $\a: F_g\xrightarrow{\cong}  F_g$ and $\mathcal{L}: \Z^k\xrightarrow{\cong}\Z^k$ are automorphisms, and
$\g:  F_g\to \Z^k$ is a homomorphism.
Then by Lemma \ref{adding lem 1}, we have
\begin{eqnarray}\label{equ 4}
\fix\phi &=&\{(u,v)\in  F_g\times \Z^k\mid u=\a(u), ~v=\g(u)+\mathcal{L}(v)\}\notag\\
&\cong& p(\fix\phi))\times \ker(p_\phi)\notag\\
&\cong& p(\fix\phi)\times\Z^s (s\leq k).\notag
\end{eqnarray}

Now we have two cases:

Case ($s<k$). In this case, by Lemma \ref{DV subgp} and Lemma \ref{iso type}, we have
$$\fix\phi\not\cong F_t\times \Z^k (t>g).$$
In particular, if $s=0$, then $\ker(p_\phi)=1$ and hence $\mathcal L$ has no eigenvalue $1$. Therefore, by Lemma \ref {adding lem 2}, we have
$\fix\phi\cong p(\fix\phi)\not\cong F_{\aleph_0}$.\\

Case ($s=k$). Note that $s=k$ holds if and only if the linear map $\mathcal{L}: \Z^k\xrightarrow{\cong}\Z^k$ is the identity.  Therefore,
\begin{eqnarray}\label{equ 8}\notag
\fix\phi &=&\{(u,v)\in  F_g\times \Z^k\mid u\in\fix\a, ~\g(u)=0\}\\
&\cong& p(\fix\phi)\times \Z^k,
\end{eqnarray}
where
$$
p(\fix\phi)=\{u\in  \fix\a \mid \g(u)=0\}\vartriangleleft \fix\a.
$$
Recall that $\g: F_g\to \Z^k$ is a homomorphism, and the image $\g(\fix\a)$ is a subgroup of $\Z^k$.

If $\g(\fix\a)$ is trivial, i.e., $\g(u)=0$ for all $u\in\fix\a$, then $p(\fix\phi)=\fix\a$. Combining Eq. (\ref{equ 8}), Proposition \ref{fixed subgp in free gp} and Lemma \ref{iso type}, we get
$$
\fix\phi=\fix\a\times \Z^k\not\cong F_t\times \Z^k (t>g), ~F_{\aleph_0}.
$$

If $\g(\fix\a)$ is nontrivial, then by Lemma \ref{adding lem 3},
$$p(\fix\phi)\cong F_{\aleph_0} \quad \mathrm{or} \quad 1,$$
and hence
$$
\fix\phi\cong p(\fix\phi)\times \Z^k\not\cong F_t\times \Z^k (t>g), ~F_{\aleph_0},
$$
where the ``$\cong$" follows from Eq. (\ref{equ 8}) while the ``$\not\cong$" follows from Lemma \ref{iso type} (recall $\aleph_0\neq t\in\N$).
\end{proof}

\subsection{Fixed subgroups in $F_g\times \Z^2$}

First, we study the aut-fixed subgroups in $F_g\times \Z^2$, which will play a key role in the proof of Theorem \ref{main thm3}.

\begin{prop}\label{key lem}
For any integer $t\geq 0$, the groups $F_t$ and $F_t\times \Z$ are aut-fixed up to isomorphism in $F_g\times \Z^2$ ($g\geq 2$).
\end{prop}

\begin{proof} Let $G=F_g\times \Z^2$ ($g\geq 2$), and let $\{a_1, a_2\ldots, a_g\}$ be a free basis of $F_g$. Every element in $\Z^2$ is denoted by a column vector.

Note that the cases of $t=0, 1$ are trivial by Lemma \ref{product auo-fixed sbgp} and Proposition \ref{fixed subgp in free gp}.
Now we suppose $t\geq 2$ and discuss in two cases.\\

(1). Construction of an automorphism $\phi_t\in\aut(G)$ such that $\fix\phi_t\cong F_t$ ($t\geq 2$).

Define an automorphism $\a\in\aut(F_g)$:
$$\a: F_g\to F_g, \quad \a(a_1)=a_1, ~\a(a_2)=a_2, ~\a(a_j)= a_j^{-1}, ~j\geq 3.$$
Then
$$
\fix\a=\langle a_1, a_2 \rangle\cong F_2.
$$
Let
$$\g: F_g\to \Z^2,\quad \g(a_1)=\begin{pmatrix}
1\\
0\\
\end{pmatrix}
, \quad \g(a_i)=\begin{pmatrix}
0\\
0\\
\end{pmatrix}, ~ i\geq 2,$$
and for any $t\geq 2$, let
$$\mathcal{L}: \Z^2\to\Z^2, \quad \mathcal{L}\begin{pmatrix}
s_1\\
s_2\\
\end{pmatrix}=\begin{pmatrix}
t&t-1\\
1&1\\
\end{pmatrix}\begin{pmatrix}
s_1\\
s_2\\
\end{pmatrix}. $$
Note that $\mathcal{L}$ is an automorphism of $\Z^2$.
Let
$$\phi_t: F_g\times\Z^2 \to F_g\times \Z^2 ,$$
$$\phi_t(u,v)=(\a(u), ~\g(u)+\mathcal{L}(v)).$$
Then $\phi_t$ is an automorphism with the inverse $\phi^{-1}_t(u, v)=\big(\a^{-1}(u), ~\mathcal{L}^{-1}(v-\g(\a^{-1}(u))\big)$.\\

Let us consider the fixed subgroup of $\phi_t$.

Note that for any $u\in F_g$, $\g(u)=\begin{pmatrix}
\nu(u,a_1)\\
0\\
\end{pmatrix}\in \Z^2$ where $\nu(u,a_1)$ is the sum of all the powers of $a_1$ in $u$. Then
\begin{eqnarray}\label{equ 1}
\fix\phi_t
&=&\left\{(u, \begin{pmatrix}
s_1\\
s_2\\
\end{pmatrix})~\Big|~u=\a(u),~\begin{pmatrix}
s_1\\
s_2\\
\end{pmatrix}=\begin{pmatrix}
\nu(u,a_1)\\
0\\
\end{pmatrix}+\begin{pmatrix}
t&t-1\\
1&1\\
\end{pmatrix}\begin{pmatrix}
s_1\\
s_2\\
\end{pmatrix}\right\}\notag\\
&=&
\left\{(u, \begin{pmatrix}
s_1\\
s_2\\
\end{pmatrix})~\Big|~u=\a(u), s_1=0, ~s_2=-\frac{\nu(u,a_1)}{t-1}\right\}\notag\notag\\
&=&\{(u, s_2)\in F_g\times \Z\mid u\in\fix\a, ~s_2=-\frac{\nu(u,a_1)}{t-1}\}.\label{eq 11}
\end{eqnarray}

Now we consider the projection $$p: F_g\times \Z\to F_g, \quad p(u, v)=u.$$
Then
\begin{equation}\label{equ 2''}
p(\fix\phi_t)=\{u\in \fix\a\mid \nu(u,a_1)\equiv 0 \mod t-1\}\cong F_{t},
\end{equation}
the above ``$\cong$" holds according to Lemma \ref{F_2 subgp}, because
$$\nu(\cdot, a_1): \fix\a=\langle a_1, a_2\rangle\twoheadrightarrow \Z$$
gives an epimorphism and the index
$$[\fix\a: p(\fix\phi_t)]=[\Z : (t-1)\Z]=t-1\geq 1.$$

Furthermore, let $p|_{\fix\phi_t}: \fix\phi_t\twoheadrightarrow p(\fix\phi_t)$ be the restriction of $p$ on $\fix\phi_t$. Since the kernel
$$\ker(p|_{\fix\phi_t})=\fix\phi_t\cap\ker(p)=(1,0)\in F_g\times \Z$$
is trivial, we have an isomorphism
$\fix\phi_t\cong p(\fix\phi_t).$
Then by Eq. (\ref{equ 2''}), we have
\begin{equation}\label{eq 12}
\fix\phi_t\cong F_{t}.
\end{equation}
the proof of Case (1) is finished.\\

(2). Construction of $\psi_t\in\aut(G)$ such that $\fix\psi_t\cong F_t\times \Z$ ($t\geq 2$).

By using the same notations as in the above case (1) but let
$$\mathcal{L}: \Z^2\to\Z^2, \quad \mathcal{L}\begin{pmatrix}
s_1\\
s_2\\
\end{pmatrix}=\begin{pmatrix}
1&t-1\\
0&1\\
\end{pmatrix}\begin{pmatrix}
s_1\\
s_2\\
\end{pmatrix},$$
Then, after the same arguments as in case (1), Eq. (\ref{eq 11}) becomes
\begin{eqnarray}
\fix\psi_t
&=&\left\{(u, \begin{pmatrix}
s_1\\
s_2\\
\end{pmatrix})\in F_g\times \Z^2~\Big|~u=\a(u),~\begin{pmatrix}
s_1\\
s_2\\
\end{pmatrix}=\begin{pmatrix}
\nu(u, a_1)\\
0\\
\end{pmatrix}+\begin{pmatrix}
1&t-1\\
0&1\\
\end{pmatrix}\begin{pmatrix}
s_1\\
s_2\\
\end{pmatrix}\right\}\notag\\
&=&
\left\{(u, \begin{pmatrix}
s_1\\
s_2\\
\end{pmatrix})\in F_g\times \Z^2~\Big|~u=\a(u), ~s_2=-\frac{\nu(u, a_1)}{t-1}\right\}\notag\\
&=&N\times \Z \notag\\
&\cong& F_{t}\times \Z,\notag
\end{eqnarray}
where $N$ is the same as ``$\fix\phi_t$" in Eq. (\ref{eq 11}):
$$
N=\fix\phi_t=\{(u, s_2)\in F_g\times \Z\mid u=\a(u), ~s_2=-\frac{\nu(u, a_1)}{t-1}\}\cong F_{t},$$
and the above ``$\cong$" follows from Eq. (\ref{eq 12}).
Note that $s_1=0$ in Eq. (\ref{eq 11}) but here $s_1$ is independent! So $s_1$ can generate a factor ``$\Z$".
\end{proof}

\subsection{Fixed subgroups in $F_g\times \Z$}

In contrast with $F_g\times \Z^k ~(k\geq 2)$, $F_g\times \Z$ contains only finitely many aut-fixed subgroups up to isomorphism. Moreover, we have a complete classification of aut-fixed subgroups in $F_g\times \Z$.

\begin{thm}\label{main thm2'}
A subgroup of $F_g\times \Z$ ($g\geq 2$) is aut-fixed up to isomorphism if and only if, it has one of the following forms:
$$F_{2t-1} ~(1\leq t\leq g), \quad F_t\times \Z^s ~(0\leq t\leq g, ~s=0,1), \quad or \quad F_{\aleph_0}\times \Z.$$
\end{thm}

According to Lemma \ref{iso type}, by removing isomorphic subgroups in Theorem \ref{main thm2'}, we obtain a direct corollary as follows.

\begin{cor}\label{main cor2'}
$F_g\times \Z$ ($g\geq 2$) contains, up to isomorphism, only the following $2g+2+\lfloor\frac g2 \rfloor$ non-isomorphic fixed subgroups of automorphisms:
$$F_m ~(0\leq m\leq g ~\mathrm{or}~ m=2t-1 \leq 2g-1),\quad F_t\times \Z (1\leq t\leq g), \quad \mathrm{or} \quad F_{\aleph_0}\times \Z.$$
\end{cor}

\begin{proof}[\textbf{Proof of Theorem \ref{main thm2'}}]
Let $G=F_g\times \Z$ ($g\geq 2$). By using the same notations as in Proposition \ref{key lem}, an automorphism $\phi\in\aut(G)$ has the form:
$$\phi(u, v)=(\a(u), \g(u)+\mathcal{L}(v)), \quad\forall (u, v)\in F_g\times \Z,$$
where $\a: F_g\xrightarrow{\cong}  F_g$ and $\mathcal{L}=\pm \id: \Z\xrightarrow{\cong}\Z$ are automorphisms, and
$\g:  F_g\to \Z$ is a homomorphism.
Then
$$
\fix\phi=\{(u,v)\in  F_g\times \Z\mid u=\a(u), ~v=\g(u)+\mathcal{L}(v)\}.
$$

Case (1). If $\g(F_g)=0$, then $\phi=\a\times \mathcal L$, and we have
$$\fix\phi=\fix\a\times \fix\mathcal L\cong F_t\times \Z^s ~~(t\leq g, s=0, 1).$$
Clearly, for any $H=F_t\times \Z^s ~~(t\leq g, s=0, 1)$, by Proposition \ref{fixed subgp in free gp},  there exists an automorphism $\phi\in\aut(G)$
such that $\fix\phi\cong H$.\\

Case (2). If $\g(F_g)\neq 0$, by considering the projection $$p:  F_g\times \Z\to  F_g, \quad p(u, v)=u,$$
and let $p_\phi: \fix\phi\to p(\fix\phi)\leq F_g$ be the restriction of $p$ on $\fix\phi$. By Lemma \ref{adding lem 1},
we have
\begin{equation}\label{eq 13}
\fix\phi\cong p(\fix\phi)\times \ker(p_\phi)
\end{equation}
where
$$p(\fix\phi)=\{u\in\fix\a \mid \exists v=\g(u)+\mathcal{L}(v)\in\Z\},$$
$$\ker(p_\phi)=\{v\in\Z\mid v=\mathcal{L}(v)\}\cong \Z^s, ~s=0,1.$$
If $\mathcal L=\id$,  then by Lemma \ref{adding lem 3} and Eq. (\ref{eq 13}), we have
$$\fix\phi\cong p(\fix\phi)\times\Z\cong F_{\aleph_0}\times \Z \quad \mathrm{or} \quad \Z.$$
If $\mathcal L=-\id$, then $\ker(p_\phi)=\{0\}$ and  Eq. (\ref{eq 13}) becomes
\begin{eqnarray}
\fix\phi&\cong& p(\fix\phi)\notag\\
&=&\{u\in  \fix\a \mid \exists v=\g(u)-v\}\notag\\
&=&\{u\in  \fix\a \mid \g(u)\equiv 0 \mod 2\},\notag
\end{eqnarray}
which is a subgroup of $\fix\a$ with index $[\fix\a :  p(\fix\phi)]\leq 2$.  By Bestvina-Handel's result (Proposition \ref{fixed subgp in free gp}), $\fix\a\cong F_t$ for some $t\leq g$. Then, by Lemma \ref{F_2 subgp}, if the index is $2$, we have
$$\fix\phi\cong F_{2t-1}, ~ 1\leq t\leq g,$$
and $\fix\phi\cong \fix\a\cong F_t$ if the index is $1$.

To summarize, $\fix\phi\cong F_{2t-1} ~(1\leq t\leq g), F_t\times \Z^s ~(0\leq t\leq g, ~s=0,1)$, or $F_{\aleph_0}\times \Z$, and all of them are aut-fixed up to isomorphism.
\end{proof}

In the above proof of Theorem \ref{main thm2'}, for any  $m\geq 1,$ if we take $\phi_m: F_g\times \Z \to F_g\times \Z$ as
$$\phi_m(u, v)=(u, \g(u)+(m+1)v),$$
$$\g: F_g\to \Z,  \quad \g(a_1)=1, ~~\g(a_i)=0, ~i\geq 2.$$
Then
\begin{eqnarray}
\fix\phi_m&=&\{(u,v)\in  F_g\times \Z\mid  v=\nu(u, a_1)+(m+1)v\}\notag\\
&=&\{(u,v)\in  F_g\times \Z\mid  v=-\frac{\nu(u, a_1)}{m}\}\notag\\
&\cong &\{ u \in  F_g \mid  \nu(u, a_1)\equiv 0 \mod m \}\notag\\
&\cong & F_{m(g-1)+1},\notag
\end{eqnarray}
where $\nu(u, a_1)$ is the sum of powers of $a_1$ in $u$. The last ``$\cong$" holds because $\fix\phi_m$ is a subgroup of $F_g$ with index $m$. Note that $\phi_m$ is not an automorphism for $m\geq 1$. Therefore, we have proved the following  (contrasting with the case of automorphisms).

\begin{prop}\label{endo-fixed sbgp}
$F_g\times \Z$ ($g\geq 2$) contains, up to isomorphism, infinitely many fixed subgroups of endomorphisms.
\end{prop}

Before providing proof of our main result, we list the following well known example for later use.

\begin{exam}\label{aleph0}
 Let $F_g\times \Z=\langle a_1, a_2\ldots, a_g\rangle \times \langle c \rangle$ for $g\geq 2$, and let
$$\phi_{\aleph_0}: F_g\times \Z\to F_g\times \Z,$$
$$a_1\mapsto a_1c, ~a_2\mapsto a_2, \quad a_j= a_j^{-1}, ~j\geq 3; \quad c\mapsto c.$$
Then $\phi_{\aleph_0}\in\aut(F_g\times \Z)$, and by straightforward computations (see also Lemma \ref{adding lem 3}), we have
$$\fix\phi_{\aleph_0}=\langle a_1^na_2a_1^{-n}\mid n\in\Z\rangle \times \langle c \rangle\cong F_{\aleph_0}\times \Z.$$
\end{exam}

\subsection{Proofs of the main results}

\begin{proof}[\textbf{Proof of Theorem \ref{main thm3}}]
To prove Theorem \ref{main thm3}, by the characterization of subgroups of $G=F_g\times \Z^k$ $(g, k\geq 2)$ (see Lemma \ref{DV subgp}) and Proposition \ref{fixed can not big than Z^k}, it suffices to prove the following subgroups are aut-fixed up to isomorphism in $G$,
\begin{enumerate}
\item $F_t\times \Z^k, ~~0\leq t\leq g;$
\item $F_t\times \Z^s, ~~t\geq 0, ~~0\leq s\leq k-1;$
\item $F_{\aleph_0}\times \Z^s, ~~1\leq s\leq k.$
\end{enumerate}

Now we verify them case by case, by using the ``multiplicativity of aut-fixed subgroup" (Lemma \ref{product auo-fixed sbgp}). Recall that every subgroup of $\Z^k$ is aut-fixed up to isomorphism. Then,

(1) Proposition \ref{fixed subgp in free gp} implies that $F_t$ ($0\leq t\leq g$) is aut-fixed up to isomorphism in $F_g$, then $F_t\times \Z^k$ is aut-fixed up to isomorphism in $F_g\times \Z^k$.

(2) For any $t\geq 0$, Proposition \ref{key lem} implies that both of $F_t$ and $F_t\times \Z$ are aut-fixed up to isomorphism in $F_g\times \Z^2$, then $F_t=F_t\times \Z^0$ is aut-fixed up to isomorphism in $F_g\times \Z^k$. Moreover, if $1\leq s\leq k-1$, then $\Z^{s-1}$ is aut-fixed up to isomorphism in  $\Z^{k-2}$, and hence
$(F_t \times \Z) \times \Z^{s-1}$
is aut-fixed up to isomorphism in $(F_g\times \Z^2)\times \Z^{k-2}$, namely,
$F_t\times \Z^s$ is aut-fixed up to isomorphism in $F_t\times\Z^k$.

(3) Note that $1\leq s\leq k$ and $F_{\aleph_0}\times \Z$ is aut-fixed up to isomorphism in $F_g\times\Z$ (see Example \ref{aleph0}), then $F_{\aleph_0}\times \Z^s$ is aut-fixed up to isomorphism in $F_g\times\Z^k$, by similar arguments as in the above cases.
\end{proof}

\section{Fixed subgroups in $\pi_1(\Sigma_g)\times \Z^k$}\label{sect 4}

In this section, we first give the main results for $\pi_1(\Sigma_g)\times \Z^k$ ($g\geq 2, ~k\geq 1$), and then show some applications in Nielsen fixed point theory.

Note that the aut-fixed subgroups of $G=\pi_1(\Sigma_g)$ or $G=F_{2g}$ are the same in form, and are both $\{F_t ~(t< 2g), ~G\}$, see Proposition \ref{fixed subgp in free gp} and Proposition \ref{fixed subgp in surface gp}. Moreover, the main idea of proofs in Section \ref{sect 3} is to use the split exact sequence
$$1\to \ker(p_\phi)\hookrightarrow\fix\phi \xrightarrow{p_\phi} p(\fix\phi)\to 1,$$
to obtain
$$
\fix\phi\cong p(\fix\phi)\times\Z^s (s\leq k),
$$
which also holds for $\pi_1(\Sigma_g)\times \Z^k$ (see Lemma \ref{adding lem 1}). So the conclusions of $\pi_1(\Sigma_g)\times \Z^k$ are almost the same as that of $F_{2g}\times \Z^k$, except for replacing subgroups $F_{m(2g-1)+1}$ with $\pi_1(\Sigma_{m(g-1)+1})$.


\subsection{Non-fixed subgroups in $\pi_1(\Sigma_g)\times \Z^k$}

\begin{prop}\label{fixed can not big than Z^k surf}
Let $\phi$ be an automorphism of $G=\pi_1(\Sigma_g)\times \Z^k$ ($g\geq 2, k\geq 1$). Then $\fix\phi$ is not isomorphic to any form below:
$$ F_{\aleph_0}, \quad F_t\times \Z^k ~(t\geq 2g), \quad  \pi_1(\Sigma_{m})\times \Z^k  ~(m> g).$$
\end{prop}

\begin{proof}
By Lemma \ref{standerd form}, an automorphism $\phi\in\aut(G)$ has a form:
$$\phi(u, v)=(\a(u), \g(u)+\mathcal{L}(v)), \quad\forall (u, v)\in \pi_1(\Sigma_g)\times \Z^k,$$
where $\a: \pi_1(\Sigma_g)\xrightarrow{\cong}  \pi_1(\Sigma_g)$ and $\mathcal{L}: \Z^k\xrightarrow{\cong}\Z^k$ are automorphisms, and
$\g:  \pi_1(\Sigma_g)\to \Z^k$ is a homomorphism.
Then by Lemma \ref{adding lem 1},
\begin{eqnarray}\label{equ 404}
\fix\phi &=&\{(u,v)\in  \pi_1(\Sigma_g)\times \Z^k\mid u\in\fix\a, ~v=\g(u)+\mathcal{L}(v)\}\\
&\cong& p(\fix\phi)\times \ker(p_\phi)\nonumber
\end{eqnarray}
where $p_\phi=p|_{\fix\phi}$ is the restriction of the projection $p:  \pi_1(\Sigma_g)\times \Z^k\to  \pi_1(\Sigma_g)$, and
\begin{eqnarray}\label{eq 405.5}
p(\fix\phi)&=&\{u\in \fix\a \mid \exists v=\g(u)+\mathcal{L}(v)\in \Z^k\}\vartriangleleft \fix\a,\\
\ker(p_\phi)&=&\{v\in\Z^k\mid v=\mathcal{L}(v)\}\cong \Z^s, \quad 0\leq s\leq k.\notag
\end{eqnarray}

Now we have two cases:

Case ($s<k$). By Lemma \ref{surface subgp} and Lemma \ref{iso type}, we have
$$\fix\phi\not\cong F_t\times \Z^k ~(t\geq 2g), ~\pi_1(\Sigma_{m})\times \Z^k ~(m> g).$$
In particular, if $s=0$, then $\ker(p_\phi)=1$ and hence $\mathcal L$ has no eigenvalue $1$. Therefore, by Lemma \ref {adding lem 2}, we have
$$\fix\phi\cong p(\fix\phi)\not\cong F_{\aleph_0}.$$

Case ($s=k$). In this case, the linear map $\mathcal{L}: \Z^k\xrightarrow{\cong}\Z^k$ is the identity, and Eq. (\ref{eq 405.5}) becomes
$$
p(\fix\phi)=\{u\in  \fix\a \mid \g(u)=0\}\vartriangleleft \fix\a.
$$

If $\g(\fix\a)$ is trivial, i.e., $\g(u)=0$ for all $u\in\fix\a$, then $p(\fix\phi)=\fix\a$. Combining Eq. (\ref{equ 404}), Proposition \ref{fixed subgp in surface gp} and Lemma \ref{iso type}, we get
$$
\fix\phi=\fix\a\times \Z^k\not\cong F_{\aleph_0}, ~F_t\times \Z^k ~(t\geq 2g), ~\pi_1(\Sigma_{m})\times \Z^k  ~(m> g).
$$

If $\g(\fix\a)$ is nontrivial, then by Lemma \ref {adding lem 3}, we have
$$
p(\fix\phi)\cong F_{\aleph_0} \quad \mathrm{or} \quad 1,
$$
and hence
\begin{equation}\label{eq 430}
\fix\phi\cong p(\fix\phi)\times \Z^k\not\cong F_{\aleph_0}, ~F_t\times \Z^k ~(t\geq 2g), ~\pi_1(\Sigma_{m})\times \Z^k  ~(m> g),\nonumber
\end{equation}
where the ``$\cong$" follows from Eq. (\ref{equ 404}) (recall $\aleph_0\neq t\in\N$).
\end{proof}

\subsection{Fixed subgroups in $\pi_1(\Sigma_g)\times \Z^k$ for $k\geq 2$}

Now we show which subgroups are aut-fixed in $G=\pi_1(\Sigma_g)\times \Z^k$ ($g, k\geq 2$).

\begin{prop}\label{thm for id surf}
For any $t\geq 0,~ m\geq 1$, and $0\leq s\leq k-1$, the following groups
$$F_t\times \Z^s, \quad \pi_1(\Sigma_{m(g-1)+1})\times \Z^s,$$
are aut-fixed up to isomorphism in $G=\pi_1(\Sigma_g)\times \Z^k$ ($g, k\geq 2$).
\end{prop}

\begin{proof}
Clearly, $F_t\times \Z^s$ $(t\leq 2g-1, s\leq k)$ is aut-fixed up to isomorphism in $G=\pi_1(\Sigma_g)\times \Z^k$ $(g, k\geq 2)$ by Proposition \ref{fixed subgp in surface gp}.
To prove the rest cases, let us pick a presentation
	$$\pi_1(\Sigma_g)=\langle a_1,b_1,\cdots ,a_g,b_g\mid [a_1,b_1]\cdots [a_g,b_g]=1 \rangle$$
and pick a basis $\{c_1,\ldots, c_k\}$ of $\Z^k$. Then an element $c_1^{s_1}c_2^{s_2}\cdots c_k^{s_k}\in \Z^k$ can be denoted by a column vector $(s_1, s_2,\ldots, s_k)^T$.\\

\textbf{Case (1).} $0\leq s\leq k-2$.

Let $\a: \pi_1(\Sigma_g)\to \pi_1(\Sigma_g)$ be an automorphism, and
define homomorphisms as follows:
 $$\g: \pi_1(\Sigma_g)\to \Z^k,$$
$$\g(a_1)=c_1, \quad \g(a_i)=\g(b_j)=1, ~ i=2,\ldots, g, ~~j=1,\ldots, g,$$
and
$$\mathcal{L}: \Z^k\to\Z^k,$$
$$\mathcal{L}(s_1, s_2,\ldots, s_k)^T=B(s_1, s_2,\ldots, s_k)^T,$$
where $B$ is a block diagonal matrix
\begin{equation}\label{da juzhen B}\nonumber
B=\begin{pmatrix}
A_\ell&0\\
0& I_{s}\\
\end{pmatrix}
\end{equation}
with $I_{s}$ ($0\leq s\leq k-2$) the identity and
$$
	A_\ell=\begin{pmatrix}
		m+1&m&m&\cdots&m&m\\
		1&1&0&\cdots&0&0\\
        1&1&1&\cdots&0&0\\
		\vdots&\vdots&\vdots&\ddots&\vdots&\vdots\\
		1&1&1&\cdots&1&0\\
		1&1&1&\cdots&1&1\\
	\end{pmatrix}
$$
an $\ell\times \ell$ matrix for some $2\leq \ell:=k-s\leq k$. Note that the first row of $A_\ell$ is $(m+1, m, \ldots, m)$, the $(i,j)$-entry ($2\leq i< j$) is $0$, and the other entries are all 1.  So $\det(B)=\det(A_\ell)=1$ and hence $\mathcal{L}$ is an automorphism of $\Z^k$.
Let
\begin{eqnarray}
\phi: \pi_1(\Sigma_g)\times\Z^k \to \pi_1(\Sigma_g)\times \Z^k ,\notag\\
\phi(u,v)=(\a(u), \g(u)+\mathcal{L}(v)).\notag
\end{eqnarray}
Then $\phi$ is an automorphism with the inverse $\phi^{-1}(u, v)=(\a^{-1}(u), ~\mathcal{L}^{-1}(v-\g(\a^{-1}(u)))$.\\

Let us consider the fixed subgroup of $\phi$.

Suppose $v=(s_1, s_2,\ldots, s_k)^T$, $\g(u)=(\nu(u,a_1), 0,\ldots, 0)^T\in\Z^k$, where $\nu(u,a_1)$ is the sum of powers of $a_1$ in $u$.
Let
$$N:=\{(u, (s_1,\ldots, s_\ell)^T)\in \fix\a\times \Z^{\ell}\mid (A_\ell-I_\ell)(s_1,\ldots, s_\ell)^T=(-\nu(u,a_1), 0, \ldots, 0)^T\}.$$
Then
\begin{eqnarray}\label{aequ 1}
\fix\phi&=&\{(u,v)\in \pi_1(\Sigma_g)\times \Z^k\mid u\in\fix\a,~ v=\g(u)+\mathcal{L}(v)\}\notag\\
&=& N\times \Z^{s}, \quad s=k-\ell.
\end{eqnarray}
Note that the linear equation
$(A_\ell-I_\ell)(s_1,\ldots, s_\ell)^T=(-\nu(u,a_1), 0, \ldots, 0)^T$
has roots
$$s_\ell=-\frac{\nu(u,a_1)}{m}, \quad s_1=s_2=\cdots s_{\ell-1}=0,$$
Therefore, we have
\begin{equation}
N=\{(u, ~(0,\ldots, 0, -\frac{\nu(u,a_1)}{m})^T)\in \fix\a\times \Z^\ell\mid ~\nu(u,a_1)\equiv 0 \mod m\}.\notag
\end{equation}

Now we consider the projection $$p: \pi_1(\Sigma_g)\times \Z^\ell\to \pi_1(\Sigma_g), \quad p(u, v)=u.$$
Then the restriction $p|_N: N\twoheadrightarrow p(N)$ of $p$ on $N$ has trivial kernel
$$\ker(p|_N)=N\cap\ker(p)=N\cap\Z^\ell=0.$$
So we have an isomorphism
\begin{equation}\label{aequ 2}
N\cong p(N)=\{u\in \fix\a\mid \nu(u,a_1)\equiv 0 \mod m\},
\end{equation}
which is a normal subgroup of $\fix\a$ with index $m\geq 1$ in the following two cases.

Therefore, if we pick $\a=\id:\pi_1(\Sigma_g)\to\pi_1(\Sigma_g)$,
then $\fix\a=\pi_1(\Sigma_g)$, and Eq. (\ref{aequ 2}) implies
\begin{equation}\label{aequ 2.2}
N\cong p(N)\cong \pi_1(\Sigma_{m(g-1)+1}),
\end{equation}
according to Lemma \ref{subgp of surface}.
Furthermore, combining Eq. (\ref{aequ 1}) and Eq. (\ref{aequ 2.2}), we have
$$\fix\phi\cong \pi_1(\Sigma_{m(g-1)+1})\times \Z^{s}.$$

If we pick $\a=\phi_1:\pi_1(\Sigma_g)\to\pi_1(\Sigma_g)$ as in Proposition \ref{fixed subgp in surface gp}, then
$\fix\a=\langle a_1, b_1 \rangle\cong F_2,$
and \begin{equation}\label{aequ 2.3}
N\cong p(N)\cong F_{m+1},
\end{equation}
because $p(N)$ is a subgroup of $\fix\a\cong F_2$ with index $m\geq 1$. Hence $\fix\phi\cong F_{m+1}\times \Z^{s}.$ Recall that $0\leq s\leq k-2$, so Case (1) is finished.\\

\textbf{Case (2).} $s=k-1$.

By using the same notations as in Case (1) but
$$
A=
\begin{pmatrix}
1&m\\
0&1\\
\end{pmatrix},
\quad
B=
\begin{pmatrix}
A&0\\
0&I_{k-2}\\
\end{pmatrix}.
$$
Then, after the same arguments as in Case (1), Eq. (\ref{aequ 1}) becomes
\begin{equation}\label{aequ 2.4}
\fix\phi=N\times \Z^{k-2},
\end{equation}
where
\begin{eqnarray}\label{aequ 2.5}
N&=&
\left\{(u, \begin{pmatrix}
s_1\\
s_2\\
\end{pmatrix})\in \fix\a\times \Z^2~\Big|~ \begin{pmatrix}
0&m\\
0&0\\
\end{pmatrix}
\begin{pmatrix}
s_1\\
s_2\\
\end{pmatrix}=
\begin{pmatrix}
-\nu(u, a_1)\\
0\\
\end{pmatrix}\right\}\notag\\
&=&
\{(u,s_2)\in \fix\a\times \Z\mid s_2=-\frac{\nu(u, a_1)}{m}, ~\nu(u, a_1)\equiv 0 \mod m\}\times \Z\notag\\
&\cong& \begin{cases}
\pi_1(\Sigma_{m(g-1)+1})\times\Z, ~ \quad \a=\id,\\
F_{m+1}\times\Z, ~ ~ \quad \a=\phi_1.
\end{cases} ,
\end{eqnarray}
where ``$\cong$'' holds by the same arguments as in Eq. (\ref{aequ 2.2}) and (\ref{aequ 2.3}). Then by Eq. (\ref{aequ 2.4}) and Eq. (\ref{aequ 2.5}), we are done.
\end{proof}

\subsection{Fixed subgroups in $\pi_1(\Sigma_g)\times \Z$}

In contrast with $\pi_1(\Sigma_g)\times \Z^k ~(k\geq 2)$, $\pi_1(\Sigma_g)\times \Z$ contains only finitely many aut-fixed subgroups up to isomorphism. Moreover, we have a complete classification of aut-fixed subgroups in $\pi_1(\Sigma_g)\times \Z$.

\begin{thm}\label{main thm2' surf}
A subgroup of $\pi_1(\Sigma_g)\times \Z$ ($g\geq 2$) is aut-fixed up to isomorphism if and only if, it has one of the following forms:
$$F_{2t-1} ~(1\leq t< 2g), \quad \pi_1(\Sigma_{2g-1}), \quad F_{\aleph_0}\times \Z,$$
$$\pi_1(\Sigma_g)\times \Z^s, \quad F_t\times \Z^s ~(0\leq t<2g, ~s=0,1).$$
\end{thm}

\begin{rem}
In \cite{Z12}, the third author investigated the relationship between the fixed subgroups and fixed points of homeomorphisms of Seifert 3-manifolds, and obtained partial results similar to Theorem \ref{main thm2' surf} by using topological methods.
\end{rem}

By directly counting the isomorphic types of subgroups in Theorem \ref{main thm2' surf}, or replace $g$ in the formula in Corollary \ref{main cor2'} with $2g$ (see the beginning part of this section for a reason), we have:

\begin{cor}\label{main cor2' surf}
$\pi_1(\Sigma_g)\times \Z$ ($g\geq 2$) contains, up to isomorphism, only $5g+2$ fixed subgroups of automorphisms.
\end{cor}

In contrast, for fixed subgroups of endomorphisms, we have:

\begin{prop}\label{endo-fixed sbgp surface}
$\pi_1(\Sigma_g)\times \Z$ ($g\geq 2$) contains, up to isomorphism, infinitely many fixed subgroups of endomorphisms.
\end{prop}

\begin{proof}[\textbf{Proof of Theorem \ref{main thm2' surf}}]
Let $G=\pi_1(\Sigma_g)\times \Z$ ($g\geq 2$). Then an automorphism $\phi\in\aut(G)$ has the form:
$$\phi(u, v)=(\a(u), \g(u)+\mathcal{L}(v)), \quad\forall (u, v)\in \pi_1(\Sigma_g)\times \Z,$$
where $\a: \pi_1(\Sigma_g)\xrightarrow{\cong}  \pi_1(\Sigma_g)$ and $\mathcal{L}=\pm \id: \Z\xrightarrow{\cong}\Z$ are automorphisms, and
$\g:  \pi_1(\Sigma_g)\to \Z$ is a homomorphism.
Then
\begin{eqnarray}\label{equ 441}\nonumber
\fix\phi=\{(u,v)\in  \pi_1(\Sigma_g)\times \Z\mid u\in\fix\a, ~v=\g(u)+\mathcal{L}(v)\}.
\end{eqnarray}

Case (1). If $\g(\pi_1(\Sigma_g))=0$, then $\phi=\a\times \mathcal L$, and
$$\fix\phi=\fix\a\times \fix\mathcal L\cong F_t\times \Z^s, ~\pi_1(\Sigma_g)\times \Z^s~~~(t<2g, s=0, 1).$$
By Proposition \ref{fixed subgp in surface gp}, for any $H$ with one of the above forms, there exists an automorphism $\phi\in\aut(G)$
such that $\fix\phi\cong H$.\\

Case (2). If $\g(\pi_1(\Sigma_g))\neq 0$, by considering the projection $$p:  \pi_1(\Sigma_g)\times \Z\to  \pi_1(\Sigma_g), \quad p(u, v)=u,$$
we have
$$
\fix\phi\cong p(\fix\phi)\times \ker(p_\phi),
$$
where $p_\phi$ is the restriction of $p$ on $\fix\phi$ and
$$p(\fix\phi)=\{u\in\fix\a \mid \exists v=\g(u)+\mathcal{L}(v)\in\Z\},$$
$$\ker(p_\phi)=\{v\in\Z\mid v=\mathcal{L}(v)\}\cong \Z^s, ~s=0,1.$$
If $\mathcal L=\id$, then $p(\fix\phi)=F_{\aleph_0}$ or trivial. We have
$$\fix\phi\cong p(\fix\phi)\times\Z\cong F_{\aleph_0}\times \Z \quad \mathrm{or} \quad \Z.$$
If $\mathcal L=-\id$, then $\ker(p_\phi)=\{0\}$ and
\begin{eqnarray}
\fix\phi \cong p(\fix\phi)\notag=\{u\in  \fix\a \mid \g(u)\equiv 0 \mod 2\},
\end{eqnarray}
which is a subgroup of $\fix\a$ with index $[\fix\a :  p(\fix\phi)]\leq 2$.  By Proposition \ref{fixed subgp in surface gp} and Lemma \ref{F_2 subgp}, $\fix\phi\cong \fix\a\cong F_t$ $(0\leq t<2g))$ or $\pi_1(\Sigma_g)$ if the index is $1$, and
$$\fix\phi\cong F_{2t-1} ~ (1\leq t< 2g), ~\pi_1(\Sigma_{2g-1}),$$
if the index is $2$.
\end{proof}

\begin{proof}[\textbf{Proof of Proposition \ref{endo-fixed sbgp surface}}]
In the proof of Theorem \ref{main thm2' surf}, for any  $m\geq 1,$ if we take $\phi_m: \pi_1(\Sigma_g)\times \Z \to \pi_1(\Sigma_g)\times \Z$ as
$$\phi_m(u, v)=(u, \g(u)+(m+1)v),$$
$$\g: \pi_1(\Sigma_g)\to \Z,  \quad \g(a_1)=1, ~~\g(a_i)=\g(b_j)=0, ~i\geq 2, ~j\geq 1,$$
where $a_i, b_j$ are the generators in the canonical presentation of $\pi_1(\Sigma_g)$.
Then
\begin{eqnarray}
\fix\phi_m&=&\{(u,v)\in  \pi_1(\Sigma_g)\times \Z\mid  v=\nu(u, a_1)+(m+1)v\}\notag\\
&=&\{(u,v)\in  \pi_1(\Sigma_g)\times \Z\mid  v=-\frac{\nu(u, a_1)}{m}\}\notag\\
&\cong &\{ u \in  \pi_1(\Sigma_g) \mid  \nu(u, a_1)\equiv 0 \mod m \}\notag\\
&\cong & \pi_1(\Sigma_{m(g-1)+1}),\notag
\end{eqnarray}
where $\nu(u, a_1)$ is the sum of powers of $a_1$ in $u$. The last ``$\cong$" holds because $\fix\phi_m$ is a subgroup of $\pi_1(\Sigma_g)$ with index $m$. Note that $\phi_m$ is not an automorphism for $m\geq 1$. We are done.
\end{proof}

\subsection{Proofs of the main results}

Now we give the main result of this section.

\begin{thm}\label{main thm in sect 4}
Let $G=\pi_1(\Sigma_g)\times \Z^k$ ($g, k\geq 2$). Then, a subgroup $H\leq G$ is aut-fixed up to isomorphism, i.e, there exists $\phi\in \aut(G)$ such that $\fix\phi\cong H$, if and only if $H$ has one of the following forms:
\begin{enumerate}
   \item $F_t\times \Z^s$ for $t\geq 0, ~0\leq s\leq k$ with $t< 2g$ if $s=k$;
  \item $\pi_1(\Sigma_{m(g-1)+1})\times \Z^s$ for $m\geq 1, ~0\leq s\leq k$ with $m=1$ if $s=k$;
  \item $F_{\aleph_0}\times \Z^s$ for $1\leq s \leq k$.
\end{enumerate}
\end{thm}

\begin{proof}
The ``only if" part follows from the subgroup classification Lemma \ref{surface subgp} and Proposition \ref{fixed can not big than Z^k surf}.

Now we prove the ``if" part. By the ``multiplicativity of aut-fixed subgroup" (Lemma \ref{product auo-fixed sbgp}),
the cases of $H=F_t\times \Z^s ~(t<2g, ~s\leq k)$ and $H=\pi_1(\Sigma_g)\times \Z^s ~(s\leq k)$ directly follow from Lemma \ref{fixed sbgp in Z^K} and Proposition \ref{fixed subgp in surface gp}. Moreover, Theorem \ref{main thm2' surf} and Lemma \ref{fixed sbgp in Z^K} imply that $F_{\aleph_0}\times \Z$ and $\Z^{s-1}$ are both aut-fixed in $\pi_1(\Sigma_g)\times \Z$ and $\Z^{k-1}$ respectively,  then the case of $H\cong F_{\aleph_0}\times \Z^s$  ($1\leq s \leq k$) holds.    The rest cases follow from Proposition \ref{thm for id surf}. Therefore, the desired conclusions hold.
\end{proof}

Take $g=2$ in item (2) of Theorem \ref{main thm in sect 4}, we have such a corollary.

\begin{cor}
For any integers $n, k\geq 2$, there exists an automorphism $\phi\in \aut(\pi_1(\Sigma_2)\times \Z^k)$ such that
$\fix\phi\cong \pi_1(\Sigma_n)$.
\end{cor}

Let's complete this section by proving Theorem \ref{main thm1} and Theorem \ref{main thm4}.

\begin{proof}[\textbf{Proof of Theorem \ref{main thm4}}]
By the subgroup classification Lemma \ref{surface subgp}, Theorem \ref{main thm4} is equivalent to Theorem \ref{main thm in sect 4} clearly.
\end{proof}

\begin{proof}[\textbf{Proof of Theorem \ref{main thm1}}]
Combining the subgroup classification lemmas (Lemma \ref{DV subgp} and Lemma \ref{surface subgp}), Theorem \ref{main thm3} and Theorem \ref{main thm4}, we have proved Theorem \ref{main thm1}.
\end{proof}

\subsection{Remark on Nielsen fixed point theory}\label{subsect Nilesen pt}

For a selfmap $f$ of a space $X$, Nielsen fixed point theory is concerned with
the properties of the fixed point set
$$\fix f:=\{x\in X\mid f(x)=x\}$$
that splits into a disjoint union of \emph{fixed point classes}. For each fixed point class $\F$ of $f$, a classical homotopy invariant \emph{index} $\ind(f,\F)\in \Z$ is defined and plays a central role in Nielsen fixed point theory. (See \cite{J83} for more information on Nielsen theory.)

In 1998, Jiang \cite{J98} posed the following question:

\begin{ques}\cite[Qusetion 3]{J98}\label{Jiang's question}
Does every compact aspherical polyhedron $X$ (i.e. $\pi_i(X)=0$ for all $i>1$) have the \emph{Bounded Index Property (BIP)} or \emph{Bounded Index Property for Homeomorphisms (BIPH)}? Namely, does there exist a bound $\B>0$ such that for any map or homeomorphism $f: X\rightarrow X$ and any fixed point class $\F$ of $f$, the index $|\ind(f,\F)|\leq \B$?
\end{ques}

In the past twenty years, many types of aspherical polyhedra had been showed to support Question \ref{Jiang's question} with a positive answer, but recently in \cite{ZZ23}, the first counterexamples were given:

\begin{thm}[\cite{ZZ23}]\label{ZZ cout-examp.} Let $S^1$ be the circle and $T^2$ the 2-dimension torus. Then
\begin{enumerate}
  \item $\Sigma_2\times S^1$ has BIPH, but does not have BIP;
  \item  $\Sigma_2\times T^2$ does not have BIPH, and hence does not have BIP.
\end{enumerate}
\end{thm}

Note that in the proof of Theorem \ref{ZZ cout-examp.} (1), the key is to construct an endomorphism $\phi_m=f_{\pi}$ of $\pi_1(\Sigma_2\times S^1)$ with index $[\pi_1(\Sigma_2): p(\fix\phi_m)]=m$ for any $m>1$. By the proof of Proposition \ref{endo-fixed sbgp surface}, such an endomorphism $\phi_m$ does exist for all $\pi_1(\Sigma_g\times S^1) (g\geq 2)$. Moreover, in the proof of Theorem \ref{ZZ cout-examp.} (2), the key is to construct an automorphism $\phi_m$ of $\pi_1(\Sigma_2\times T^2)$ with index $[\pi_1(\Sigma_2): p(\fix\phi_m)]=m$ for any $m>1$. By the proof of Proposition \ref{thm for id surf}, such an automorphism $\phi_m$ also exists for all $\pi_1(\Sigma_g\times T^k) (g, k\geq 2)$. Therefore, by using the same arguments as in \cite{ZZ23}, we can extend Theorem \ref{ZZ cout-examp.} to higher genus and dimension as follows.

\begin{thmbis}{high-dim cout-examp.}
Let $S^1$ be the circle and $T^k$ the $k$-dimension torus. Then
\begin{enumerate}
  \item $\Sigma_g\times S^1 (g\geq 2)$ has BIPH, but does not have BIP;
  \item $\Sigma_g\times T^k (g, k\geq 2)$ does not have BIPH, and hence does not have BIP.
\end{enumerate}
\end{thmbis}

\section{Some results of hyperbolic groups}\label{sect 5}

More generally, we can consider the direct products of hyperbolic groups and free-abelian groups. First, we have the following results.

\begin{prop}\label{fixed subgp hyper gp Prop.}
Let $G=H\times \Z^k (k\geq 2)$ for $H$ a non-elementary torsion-free hyperbolic group, and $\b_1(H)$ the first betti number of $H$.
\begin{enumerate}
  \item If $\b_1(H)=0$, then every aut-fixed subgroup of $G$ has a form $N\times \Z^s (s\leq k)$ for $N\subseteq H$ an aut-fixed subgroup of $H$;
  \item If $\b_1(H)\geq 1$, then for any $m\geq 1$ and any $0\leq s<k$, there exists $\phi\in \aut(G)$ such that $$\fix\phi\cong N\times \Z^s $$ where $N$ is a normal subgroup of $H$ with index $[H:N]=m$.
  \end{enumerate}
\end{prop}


\begin{proof}
(1) If $\b_1(H)=0$, then there is no nontrivial homomorphism from $H$ to $\Z^k$. By Lemma \ref{standerd form}, every automorphism $\phi: G\to G=H\times \Z^k$ has a form
$$\phi(u, v)=(\a(u),\mathcal L(v)),$$
where $\a: H\xrightarrow{\cong} H$ and $\mathcal{L}: \Z^k\xrightarrow{\cong}\Z^k$ are automorphisms. Hence
$$\fix\phi=\fix\a\times \fix\mathcal L=N\times \Z^s (s\leq k),$$
where $N=\fix\a\subseteq H$ as requested.\\

(2) Now we suppose $\b_1(H)\geq 1$. Then there exists an epimorphism $\gamma: H\to \Z$, that can be viewed as $\gamma: H\to \Z\times \Z^{k-1}, h\mapsto (\g(h), 0)$.
By using the ``multiplicativity of aut-fixed subgroups" as in the proof of Theorem \ref{main thm3}, it suffices to prove the conclusion for the case $k=2$.

For any $m\geq 1$ and $0\leq s<2$, we can define automorphisms
\begin{eqnarray}
\phi_s: H\times\Z^2 \to H\times \Z^2 ,\notag
\quad \phi_s(u,v)=(u, \g(u)+\mathcal{L}_s(v)),
\end{eqnarray}
where the automorphisms
$$\mathcal{L}_0=\begin{pmatrix}
		m+1 & m\\
		1 & 1\\
	\end{pmatrix}, ~ \mathcal{L}_1=\begin{pmatrix}
		1 & m\\
		0 & 1\\
	\end{pmatrix}: \Z^2\to\Z^2.$$
Then
	\begin{eqnarray}
		\fix\phi_0&=&\{(u,\begin{pmatrix}
		v_1\\
		v_2\\
	\end{pmatrix})\in H\times \Z^2\mid \begin{pmatrix}
		v_1\\
		v_2\\
	\end{pmatrix}=\g(u)+\mathcal{L}_0\begin{pmatrix}
		v_1\\
		v_2\\
	\end{pmatrix}\}\notag\\
&=&\{(u,\begin{pmatrix}
		0\\
		v_2\\
	\end{pmatrix})\in H\times \Z^2\mid \g(u)=-mv_2\}\notag\\
&\cong& \{u\in H \mid ~\g(u)\equiv 0 \mod m\}:=N,\nonumber
\end{eqnarray}
where $N$ is a normal subgroup of $H$ with index $[H:N]=m\geq 1$, and
\begin{eqnarray}
		\fix\phi_1&=&\{(u,\begin{pmatrix}
		v_1\\
		v_2\\
	\end{pmatrix})\in H\times \Z^2\mid \begin{pmatrix}
		v_1\\
		v_2\\
	\end{pmatrix}=\g(u)+\mathcal{L}_1\begin{pmatrix}
		v_1\\
		v_2\\
	\end{pmatrix}\}\notag\\
&=&\{(u,\begin{pmatrix}
		v_1\\
		v_2\\
	\end{pmatrix})\in H\times \Z^2\mid \g(u)=-mv_2\}\notag\\
&\cong& N\times \Z,\nonumber
\end{eqnarray}
here the factor $\Z$ is generated by the independent element $v_1$. We are done.
\end{proof}


A group G is called \emph{finite index rigid} if it does not contain isomorphic finite index subgroups of different
indices. For example, the cyclic group $\Z$ is not finite index rigid
(e.g. $2\Z\cong 3\Z$). In contrast, finite groups and groups with non-zero Euler characteristic (e.g.
non-abelian free groups and hyperbolic surface groups) are finite index
rigid. For more finite index rigid groups, see Lazarovich's recent paper \cite{Laz23}. Moreover, Lazarovich showed:

\begin{thm}[Lazarovich]
Every non-elementary hyperbolic group is finite index rigid.
\end{thm}

\begin{proof}[\textbf{Proof of Theorem \ref{main thm5}}]
Recall that every hyperbolic group contains, up to isomorphism, only finitely many fixed subgroups of automorphisms. Moreover, since the non-elementary torsion-free hyperbolic group $H$ is finite index rigid, its subgroups of distinct indices are non-isomorphic. Then the the conclusion follows form Proposition \ref{fixed subgp hyper gp Prop.} clearly.
\end{proof}

\noindent\textbf{Acknowledgements.} The authors are very grateful to the anonymous referee, whose valuable and detailed comments greatly improved our initial manuscript. The authors also thank Qian Chen and Ke Wang for their valuable communications.

\appendix
\setcounter{table}{0} 
\setcounter{figure}{0} 
\setcounter{equation}{0} 
\renewcommand{\theequation}{A.\arabic{equation}} 




\end{document}